\documentclass[12pt,reqno]{amsart}
\usepackage[utf8]{inputenc}
\usepackage{amsmath}
\usepackage{amsfonts}
\usepackage{amssymb}
\usepackage{amsthm}
\usepackage[left=3cm,right=3cm,top=3cm,bottom=3cm]{geometry}

\usepackage{hyperref} 

\newtheorem{theorem}{Theorem}
\newtheorem{lemma}{Lemma}
\newtheorem{definition}{Definition}
\newtheorem{proposition}{Proposition}

\newcommand{\f}{\varphi}

\newcommand{\na}{\nabla}
\newcommand{\A}{\mathcal{A}}

\newcommand{\N}{{\mathbb{N}}}

\newcommand{\D}{\partial}
\newcommand{\tr}{{\rm Tr}}

\begin{document}
\title{Convergence of the Hesse-Koszul flow on compact Hessian manifolds}

\author{St\'ephane Puechmorel and Tat Dat T\^o  
}
\date{\today}

\address{Ecole Nationale de l'Aviation Civile, Unversit\'e F\'ed\'erale de Toulouse\\
7, Avenue Edouard Belin\\
FR-31055 Toulouse Cedex}

\address{Institut Math\'ematiques de Toulouse \\ Universit{\'e} de Toulouse, CNRS, UPS
\\ 31062 Toulouse Cedex 09\\ France (Associated Researcher).}
\email{stephane.puechmorel@enac.fr}
\email{tat-dat.to@enac.fr}

\maketitle              
\begin{abstract}
We study the long time behavior of the Hesse-Koszul flow on  compact Hessian manifolds. When the first  affine Chern class is negative, we prove that the flow converges to the unique  Hesse-Einstein metric. We also derive a convergence result for a twisted Hesse-Koszul flow on any compact Hessian manifold. These results give  alternative proofs for the existence of the unique Hesse-Einstein metric by Cheng-Yau and Caffarelli-Viaclovsky  as well as the real Calabi theorem by  Cheng-Yau, Delano\"e and Caffarelli-Viaclovsky. 
\end{abstract} 
\section{Introduction}
An affine  manifold is  a real manifold $M$ which admits a flat, torsion free connection $\na$ on its tangent bundle. A Riemannian metric $g$ on an affine manifold $(M,\nabla)$ is called a {\sl Hessian metric} if $g$ can be locally expressed by $g=\nabla d\varphi $.  We then say $(M,\na, g)$ is a Hessian manifold. We consider the Monge-Amp\`ere operator:
\begin{equation}
   MA(\f):= \det(\na d \f)=\det \left( \frac{\partial^2 \varphi}{\partial x^i\partial x^j}\right)
\end{equation}
where $\{x^1,\ldots,x^n\}$ is an affine coordinate system with respect to $\nabla$. It was observed by 
Cheng-Yau \cite{CY} that this is a natural operator on affine manifolds since it is invariant under affine coordinate transformations. In particular, it is very similar to the complex Monge-Amp\`ere   operator since the  (real) Monge-Amp\`ere measure as $\mu_\f= \sqrt{\det(\f_{ij})} dx^1\wedge \ldots\wedge  dx^n$ is well-defined.  We refer the interested readers to \cite{Dui, Kos61, Kos62, Kos65, Kos68, CY, Shi, Tot, CV} and references therein for more details  on Hessian manifolds.  

\medskip
Hessian  manifolds and real Monge-Amp\`ere equations play a central role in many fields varying from mathematical physics to  statistics. They appear as  a large complex limits of  Calabi-Yau manifolds, which is in the framework of the Strominger-Yau-Zaslow \cite{SYZ},  and  the  Kontsevich-Soibelman \cite{KS} conjectures (see also \cite{As, GW,Hit, LYZ}).

  Recently Hessian manifolds have been interpreted as particular parameter spaces of  statistical models in which the Fisher-Rao metric is a Hessian metric (cf. \cite{AN,Ama,Ay,Lau}).  Studying  geometric structures of Hessian manifolds thus could lead to many applications in statistics. 

\medskip
In \cite{CY}, Cheng-Yau study the real Monge-Amp\`ere equation
\begin{eqnarray}\label{MAE}
 \det(g+\na d \f)= e^{\lambda \f +f }  \det{g},
\end{eqnarray}
with either $\lambda =0$ or $\lambda >0$.  When $\lambda=0$, solving this equation provides  a solution to  the real Calabi problem (cf. \cite{CY}): given $\eta\in c_1^a (M)$,  it shows that there is  a metric $\tilde{g}=g+\na d \f$ for some $\f\in C^\infty(M)$ such that $\kappa (g)=\eta$, where $\kappa(g)$ is the second Koszul form (see Definition \ref{def:Koszul_form}).  When $\lambda >0$, solving  this equation allows one to construct:   Hesse-Einstein metric $\tilde g=g+\na d \f$, i.e $\kappa(g)=-\lambda g$ which is a canonical metric on Hessian manifolds.  Hesse-Einstein metrics in Hessian geometry can be seen as the real version of the K\"ahler-Einstein metrics  in K\"ahler geometry (cf. \cite{CY,Lof02,Lof09, Shi})

\medskip
Assuming that $M$ is a special manifold Cheng-Yau \cite{CY} solved the equation \eqref{MAE} with $\lambda=0$ using the continuity method. Delano\"e then removed this condition in  \cite{Del}  but still relied on \cite{CY} for higher derivative estimates. In \cite{CY}, the authors also solved \eqref{MAE} with $\lambda>0$  by lifting this equation to $M+i\mathbb{R}^n$ and using methods from complex geometry. Caffarelli-Viaclovsky \cite{CV} then generalized these previous works solving  \eqref{MAE} assuming a minimal regularity for $f$.  They used the continuity method for $\lambda=0$ and the viscosity method for $\lambda >0$. We also refer to \cite{HM} for a recent variational approach with optimal transport point of view.

In this note, we  give an alternative approach using a geometric flow, namely the Hesse-Koszul flow.   
This flow has been defined and studied by Mirghafouri-Malek \cite{MM} on compact Hessian manifolds. Given  any compact Hessian manifold  $(M, \nabla,  g_0=\nabla d \psi )$ we define the following evolution equation 

\begin{equation}\label{HEF}
    \frac{\partial  g_{ij}}{\partial t}=-\beta_{ij}(g), \quad g|_{t=0}= g_0
\end{equation}
where $\beta(g) =-2 \kappa(g)$, with $\kappa$  the second Koszul form for $(\nabla, g)$ (see Definition \ref{def:Koszul_form}). Along the flow the evolved metric $g$ remains Hessian, this is why we call it the Hesse-Koszul flow. 

Similarly to the K\"ahler-Ricci flow,  we can rewrite the flow as a sacalar equation, namely the  parabolic 
Monge-Amp\`ere equation:

\begin{equation}\label{PMAE_0}
    \frac{\D }{\D t }\f= \log\dfrac{\det( \hat g(t)+\na d \f)}{\det g_0}+ f, \quad \f_0 =0
\end{equation}
where $\f$ is the unknown function and $\hat g (t)$ only depends on $t$, $g_0$ and the first affine Chern class $c_1^a(M)$ (see Definition \ref{first_aff_chern_class}).

\medskip
In \cite{MM}, the authors proved  the  short-time existence and the uniqueness of the flow on compact Hessian manifolds. When $\hat g$ is independent of $t$, they showed that the flow has a long time existence.  In this paper, we study the characterization of the maximal existence time and the long time behavior of the flow. 

Our first goal is to prove that the  the maximal  time for the existence of smooth solution is a cohomological constant like the one of the K\"ahler-Ricci flow (see for instance \cite{Cao,SW,Tos,TZ,Wei}):

 \begin{theorem}\label{thm:maximal_time_0}
Let $(M,\na, g)$ be a compact Hessian manifold. Then the Hesse-Koszul flow has unique smooth solution $g(t)$ on the maximal time interval [0,T), where
\begin{equation}
    T=\sup \{ t>0|\, [g_0]- tc^a_1(M)>0 \}.
\end{equation}
 \end{theorem}
For the proof we adapt some  K\"ahler-Ricci flow techniques to our case.  There is indeed a natural connection between Hessian and K\"ahler geometry, as first observed by Dombrowski \cite{Dom}: the tangent bundle over a Hessian manifold admits a K\"ahler metric induced by the Hessian metric.

Our second goal is to prove that the flow converges to a Hesse-Einstein metric  assuming that the first affine Chern class  is  negative. This gives an alternative proof for the result in  \cite{CY,CV} on the existence of Hesse-Einstein metrics:

\begin{theorem} Let $M$ be a compact Hessian manifold. 
 Assume  that $c_1^a(M)<0$, then starting  from any Hessian metric $g_0$, the normalized  Hesse-Koszul flow  
 $$
 \frac{\partial  g_{ij}}{\partial t}=-\beta_{ij}(g)-g$$
 exists for all time and converges in $C^\infty$ to a Hesse-Einstein metric $g_{\infty}$ satisfying
\begin{equation}\label{eq:HE_neg_0}
\beta(g_{\infty})=-g_{\infty}.
\end{equation} 
Moreover, $g_\infty$ is the unique solution to the Hesse-Einstein equation \eqref{eq:HE_neg_0}. 
\end{theorem}

Finally we give another proof for  a real version of  Calabi's conjecture due to \cite{CY,Del, CV}. We follow   Cao's approach (cf. \cite{Cao}) to the Calabi conjecture to  study the Hesse-Koszul flow twisted by  $\eta \in c_1^a(M)$:

\begin{theorem}\label{thm:conv_calabi_0}
Let $M$ be a compact Hessian manifold. 
The flow  
\begin{equation}\label{eq:THKF_0}
     \frac{\partial  g_{ij}}{\partial t}=-\beta_{ij}(g) + \eta, \quad g|_{t=0}=\hat g.
\end{equation}
 exists for all time and $C^\infty$-converges  to a metric $g_\infty $ which is the unique solution to 
\begin{equation}\label{eq:Calabi_0}
    \beta(g_\infty)=\eta.
\end{equation}

\end{theorem}
The long time existence is due to \cite{MM}.  For the proof of the convergence, we  derive uniform {\sl a priori} estimates by   adapting some K\"ahler-Ricci flow  techniques to our case, and by using a   new approach based on  \cite{PT} to prove the  $C^0$-estimate. 

\medskip
Finally we show that our approach can be applied to prove a  convergence result  for a parabolic Monge-Amp\`ere equation on compact Riemannian manifolds: 

\begin{theorem} \label{PMA_Rieman_0} Let $(M,g)$ be a compact Riemannian manifold and $\na $ be the Levi-Civita connection of $g$. 
The normalization  $\tilde{\f}:=\f-\frac{1}{Vol_g}\int_M \f dV_g$  of the solution  for  the flow 
\begin{equation}\label{PMA_riemannian}
\frac{\D }{\D t}\f (x,t)=\log \frac{\det( g(x)+ \na^2 \f(t,x) )}{\det g(x)}- f(x),
\end{equation}
$C^\infty$-converges   to a function $\tilde \f_\infty$.  In particular, the limit $\tilde{\f}_\infty$ is  a solution 
of the following Monge-Amp\`ere equation:
\begin{equation}\label{MA_Rieman}
\det(g+\na^2 \phi )=ce^f \det (g), 
\end{equation}
for some constant $c$.
\end{theorem}
This flow was studied in \cite{Hui} where a long time existence was proved. We establish further  the convergence  of the flow. 
Our key ingredient of the proof is a uniform $C^0$ estimate for the normalization of $\f$ as in Theorem \ref{thm:conv_calabi_0}.  
Moreover  Theorem  \ref{PMA_Rieman_0} gives an alternative proof to the existence of  solutions for  the Monge-Amp\`ere equation \eqref{MA_Rieman} on compact Riemannian manifolds due to \cite{Del_JFA}.

\medskip
\noindent\textbf{Acknowledgement.} The  authors are grateful to  Vincent Guedj for  suggestions and encouragement.  We also would like to thank Yuxin Ge and  Hoang Chinh Lu for useful discussions.

\section{Preliminary}

\subsection{Affine manifolds, Hessian metric and Koszul forms}
\begin{definition}\cite{Shi}
An {\sl affine manifold} $(M,\nabla)$ is a differentiable manifold equipped with a flat, torsion-free connection $\nabla$.

A Riemannian metric $g$ on an affine manifold $(M,\nabla)$ is called a {\sl Hessian metric} if $g$ can be locally expressed by $g=\nabla d\phi  $. Then  $(M,\nabla, g)$ is called a {\sl Hessian manifold}.
\end{definition}
It is known that  a manifold $M$ is affine if and only if $M$ admits an affine atlas such that transition functions are  in the affine group ${\rm Aff}(n)=\{ \Phi: \mathbb{R}^n\rightarrow \mathbb{R}^n, \Phi(x)=Ax+b \}$.  

Let $(M,\nabla, g)$ be a Hessian manifold,
$g$ can be locally expressed by 
$$g_{ij}= \frac{\partial^2  \phi}{\partial x^i\partial x^j}$$
where $\{x^1,\ldots,x^n\}$ is an affine coordinate system with respect to $\nabla$.  

Denote $\hat \nabla$ the Levi-Civita of $(M,g)$, $\gamma= \hat \nabla- \nabla $. Since $\na$ and $\hat \na$ are torsion-free, we have
$$\gamma_X Y= \gamma _Y X.$$
Moreover, the components $\gamma^i{}_{jk}$ of $\gamma$ with respect to affine coordinate systems coincide with the Christoffel symbols $\Gamma ^i{}_{jk}$ of the Levi-Civita connection $\hat \nabla$.

The  tensor $Q=\nabla \gamma$ is  called the {\sl Hessian curvature tensor} for $(g,\nabla)$. We recall here some properties of the Hessian manifolds. 
\begin{proposition}\cite{Shi}
Let $(M, \nabla)$ be an affine manifold and $g$ a Riemannian metric on $M$. Then the following are equivalent:
\begin{enumerate}
    \item $g$ is a Hessian metric
    \item $(\nabla_X g)(Y,Z)=(\nabla_Y g)(X,Z)$
    \item $\dfrac{\D g_{ij} }{\D x^k}=\dfrac{\D g_{kj}}{\D x^i}$
    \item $g(\gamma _X Y, Z)= g(Y, \gamma _X Z)$
    \item $ \gamma_{ijk}=\gamma_{jik}$
\end{enumerate}
\end{proposition}

\begin{proposition}\cite{Shi}
Let $\hat R$ be the Riemannian curvature 
of $g=\na d \phi$ and $Q=\nabla \gamma$ is   the Hessian curvature tensor for $(g,\nabla)$. Then

\begin{enumerate}
\item $Q_{ijkl}= \dfrac{1}{2}\dfrac{\D ^4 \phi}{\D x^i\D x^j \D x^k \D x^l}-\dfrac{1}{2}g^{pq}\dfrac{\D ^3 \phi}{\D x^i \D x^k \D x^p} \dfrac{\D ^3 \phi}{\D x^j \D x^l \D x^q}$
    \item $\hat R(X,Y) = -[\gamma_X, \gamma_Y], \quad \hat R^{i}{}_{jkl} = \gamma ^i{}_{lm} \gamma^m{}_{jk} -\gamma^i{}_{km}\gamma^m{}_{jl}.$
    \item $\hat R_{ijkl}= \dfrac{1}{2} (Q_{ijkl}-Q_{jikl})=-  \dfrac{1}{4}\phi^{pq} ( \phi_{ikp}\phi_{jlq}-\phi_{jkp}\phi_{ilq})
 $,\\
 where $\phi_{ikp}:=\dfrac{\D ^3 \phi}{\D x^i \D x^k \D x^p}$ and $(\phi^{pq})=(\phi_{pq})^{-1}$.
\end{enumerate}

\end{proposition}

\begin{definition}\label{def:Koszul_form} We define  first Koszul form $\alpha$ and the the second Koszul form $\kappa$ for $(\nabla, g)$ (cf. \cite{Kos61, Shi}) by
$$\nabla_X vol_g=\alpha(X) vol_g\quad {\rm and} \quad  \kappa=\nabla \alpha.$$
\end{definition}
It follows from the definition that \begin{eqnarray}
\alpha(X)&=&{\rm Tr} \gamma_X \\
\kappa (g) &= & \frac{1}{2} \nabla d (\log \det g), 
\end{eqnarray}
so 
\begin{eqnarray}
\alpha_i &=& \frac{1}{2}\dfrac{ \D \log \det [g_{pq}]}{\D x^i} = \gamma^k{}_{ki}\\
\kappa_{ij}&=&   \frac{1}{2}\dfrac{ \D^2 \log \det [g_{pq}]}{\D x^i\D x^j}. 
\end{eqnarray}
In the sequel we shall use the the tensor $\beta=-2\kappa$ instead of $\kappa$ to define the Hesse-Koszul flow.  

 We shall use the following  operator 
$L_g(f):= \tr_g \na df $ for any Riemannian metric $g$. In an affine coordinate system $\{x^1, \ldots,x^n\}$  with respect to $\na$, we have
\begin{eqnarray}
    L_g(f)&=&\tr_g\hat \na df  - \tr \gamma df\\
    &=& \Delta_g f   -g^{ij}\gamma^k{}_{ij}\partial_k f , 
\end{eqnarray}
where $\hat \na $ is the Levi-Civita connection of $g$ and $\gamma= \hat \na -\na $. $L$ is an elliptic operator and the maximum principle holds for $L$.
\subsection{Cohomology on affine manifolds and the first affine Chern class}
Let $(M,\na)$ be an affine manifold. Denote by 
$(  \wedge^p T^*M)\otimes(\wedge^q T^*M$), the tensor product of vector bundles   $  \wedge^p T^*M$ and $\wedge^q T^*M$.
Denote by $\A^{p,q}$ all smooth sections of $(  \wedge^p T^*M)\otimes (\wedge^q T^*M$). In an affine coordinate system with respect to $\na$, a $(p,q)$-form $\omega$ in $\A^{p,q}$ is expressed by
\begin{eqnarray*}
\omega &=& \sum \omega_{i_1 \ldots
 i_p; j_1\ldots j_q \ }(dx^{i_1}\wedge\ldots\wedge dx^{i_p})\otimes (dx^{j_1}\wedge\ldots\wedge dx^{j_q}) \\
 &=& \sum \omega_{I_p;J_q}dx^{I_p}\otimes dx^{J_q} \quad \text{with } I_p=(i_1, \dots, i_p), \, J_q=(j_1,\dots,j_q)
\end{eqnarray*}
where $\omega_{I_p;J_q}=\omega_{i_1 \ldots
 i_p; j_1\ldots j_q \ }$,  $dx^{I_p}=dx^{i_1}\wedge\ldots\wedge dx^{i_p}$.
 \begin{definition}
 For $\alpha \in \A^{p,q}$ and $\beta$ in $\A^{r,s}$, we define the exterior product
 $\alpha\wedge \beta \in \A^{p+r, q+s}$ by
 \begin{eqnarray}
 &&(\alpha\wedge \beta ) (X_1,\ldots,X_{p+r};Y_1,\ldots,Y_{q+s}) \\
&& \hspace{1cm} = \frac{1}{p!q!r!s!}\sum_{\sigma,\tau}\epsilon_\sigma \epsilon_\tau \alpha(X_{\sigma(1)},\ldots X_{\sigma  (p)}; Y_{\tau(1),}\ldots , Y_{\tau(q)})\\
  & & \hspace{3.5cm}\times \beta(X_{\sigma(p+1)},\ldots X_{\sigma  (p+r)}; Y_{\tau(q+1),}\ldots , Y_{\tau(q+s)}),
 \end{eqnarray}
 where the sum is taken over all permutation $\sigma, \tau$ and $\epsilon_\sigma$ (resp. $\epsilon_\tau$) is the sign of $\sigma$ (resp. $\tau$). 
 \end{definition}
 \begin{definition}
 We define $d'= d^\na \otimes I : \A^{p,q}\rightarrow \A^{p+1,q}$ and $d''=I \otimes d^\na: \A^{p,q}\rightarrow \A^{p,q+1}$, where $I$ is the identity operator and  $d^\na$ is the exterior derivative induced by $\na$.  
 \end{definition}

\begin{definition} \cite{CY,Shi}
We define a cohomology group $$\tilde{H}^k(M) = \{ \alpha\in \A^{p,q}| d'\alpha=0, d''\alpha=0\}/d'd''(\A^{k-1,k-1}).$$
\end{definition}

Let  $\Omega=\Omega(x) dx^1\wedge\ldots\wedge dx^n $ be a volume form on $M$.  Then the second Koszul form of $\Omega$ is defined by 
$$\kappa_\Omega=\sum  \frac{\D^2 \log \Omega(x)}{\D x^i\D x^j}dx^i\otimes dx^j.$$
Again we shall use the tensor $\beta_\Omega:=-2\kappa_\Omega$ to define the Hesse-Koszul flow. 
Denote by $[\kappa_\Omega]\in \tilde{H}^1(M)$ the class represented by $\kappa_\Omega$. If $\Omega'$ is another volume form, then there exists a function $f$ on $M$ such that $\Omega' = e^f \Omega$, so we have
$$\kappa_\Omega=\kappa_{\Omega'}+ \na d f.$$ By definition, we have
 $[\kappa_\Omega]=[\kappa_{\Omega'}]\in \tilde{H}^1(M)$, so we can define  the {\sl first affine Chern class} as follows:
\begin{definition}\label{first_aff_chern_class}
We set $c_1^a(M) :=-2[\kappa_\Omega]=[\beta_\Omega]\in \tilde{H}^1(M) $ to be the {\sl first affine Chern class}  of $M$, for any volume form $\Omega$. 
\end{definition} 
In particular,  if $(M,\na, g)$ is a  compact Hessian manifold then  $c^a_1(M)=-2[\kappa(g)]=[\beta(g)]$, where $\kappa(g)$ is the second Koszul form of $(\na, g)$ (see Definition  \ref{def:Koszul_form}).

Let $[\alpha]\in \tilde{H}^1(M)$, we say that $[\alpha]$ is positive (resp. semi-positive) and denote $[\alpha]>0$ (resp. $[\alpha]\geq 0$) if there exists $\alpha'\in [\alpha]$ such that $\alpha'>0$ (resp. $\alpha'\geq 0$).   Then we have the following theorem due to Shima \cite{Shi} (see also Delano\"e \cite{Del})
\begin{theorem}
Let $(M,\na, g)$ be compact Hessian manifold and $\alpha$ and $\kappa$ be the first and the second Koszul forms respectively. Then we have
\begin{itemize}
\item[(i)]  $$\int_M \tr_g \kappa dV_g=\int_M \|\alpha\|^2 dV_g\geq 0.$$ 
\item[(ii)] If $\int_M \tr_g \kappa dV_g=0$ then the Levi-Civita connection $\hat \na $ of $g$ coincides with $\na$.
\end{itemize}
In particular, the first affine Chern class $c_1^a(M)$ cannot be positive.
\end{theorem}

\section{Maximal existence time for the  flow on compact manifolds}
Let $(M,\na, g)$ be a compact Hessian manifold of dimension $n$. Consider the Hesse-Koszul flow
\begin{equation}\label{HKF_max}
    \frac{\partial  g_{ij}}{\partial t}= 2\kappa (g), \quad g|_{t=0}= g_0.
\end{equation}
For our convenience,  we shall write the 
Hesse-Koszul flow  as
\begin{equation}\label{HKF_max}
    \frac{\partial  g_{ij}}{\partial t}= -\beta (g), \quad g|_{t=0}= g_0,
\end{equation}
where $\beta=-2\kappa$ represents the first affine Chern class. 
In this section we prove the the maximal existence time for the Hesse-Koszul flow is a cohomological constant.  It only depends  on the first affine Chern class. 
Define 
\begin{equation}
    T=\sup \{ t>0|\, [g_0]- tc^a_1(M)>0 \},
\end{equation}
then the main result of this section is the following:
\begin{theorem}\label{thm:maximal_time}
Let $(M,\na, g)$ be a compact Hessian manifold. Then the Hesse-Koszul flow has unique smooth solution $g(t)$ on the maximal time interval [0,T).
 \end{theorem}
  We  follow the approach developed in K\"ahler geometry, to establish a  similar result for the K\"ahler-Ricci flow (cf. \cite{Cao,SW,Tos,TZ,Wei}). 
 \subsection{Reduction to parabolic Monge-Amp\`ere equation}
 
 Fix any $T'<T$, our goal is to  show that there exists a solution  for the flow \eqref{HKF_max} on $[0,T')$. The key ingredient is that we can rewrite the Hesse-Koszul flow \eqref{HKF_max} as a parabolic Monge-Amp\`ere equation.
 
Since $[g_0]-T'c_1^a(M)>0$, there exists a Hessian metric $g'\in  [g_0]-T'c_1^a(M) $, then 
$\frac{1}{T'}(g'-g_0)\in -c_1^a(M) $. Therefore 
$$\hat g(t) = g_0+ \frac{t}{T'}(g'-g_0) = \frac{
1  }{T'}((T'-t)g_0+ tg') $$
 is also a Hessian metric for all $t\in [0,T']$.  Fix $\Omega_0$ a smooth positive volume form on $M$. Since $\beta_{\Omega_0}\in c_1^a(M)$,  there exist a function $f\in C^{\infty}(M)$ such that
 $$\frac{1}{T'}(g'-g_0)= -\beta_{\Omega_0}+ \na d f.$$
 Then we define $\Omega
 = e^{f/2}\Omega_0$ which is a smooth positive volume form satisfying 
 $$\beta_\Omega = - \frac{1}{T'}(g'-g_0).$$
 By  abuse of notation, we shall write $\Omega$ as the local density as well, i.e  $$\Omega=\Omega(x)dx^1\wedge \ldots\wedge dx^n,$$
 and $\Omega^2 :=\Omega^2(x)$ the square of the density. 
 Then we have the following:
 \begin{lemma}\label{lem:reduction_max}
 A smooth family of Hessian metric $g_t$ on $[0,T')$ is the solution of the Hesse-Koszul flow \eqref{HKF_max} if and only if the parabolic equation
\begin{equation}\label{PMA_max}
    \dfrac{\D  }{\D t} \f =\log \dfrac{\det (\hat g_t + \na d \f) }{ \Omega^2}, \quad \tilde g_t + \na d \f>0, \quad \f|_{t=0}=0
\end{equation}
has a smooth solution $\f(t), t\in [0,T')$ such that $g_t= \hat g_t+\na d \f(t)$.
 \end{lemma}
 
 \begin{proof}
 If $\f(t)$ satisfies the equation \eqref{PMA_max}, then we set $g(t)=\hat g(t)+ \na d \f(t)$. By a straightforward computation, we get 
 \begin{eqnarray}
 \frac{\D }{\D t}g(t) =  -\beta_{\Omega}+ \na d\log \frac{\det g(t)}{\Omega^2} = -\beta_{\Omega} + \beta_\Omega -\beta(g(t)) =-\beta(g(t)).  
 \end{eqnarray}
 Since $g(0)= g_0$, we imply that $g(t)$ is a solution of \eqref{HKF_max}.
 
 \medskip
 For the ``only if" assertion, given $g(t)$ a solution of \eqref{HKF_max}, we define for any  $t\in [0,T') $
 $$\f(t)=\int_0^t\log \frac{\det g(s)}{\Omega^2}ds.$$
 Then we get
 \begin{eqnarray}
\frac{\D }{\D t}\f(t)= \log \frac{\det g(t)}{\Omega^2},\quad \f(0)=0. 
 \end{eqnarray}
 We now prove that $g(t)= \hat g(t)+\na d \f(t)$ on $[0,T')$. 
 Since $$\na d \dot\f(t) =\na d  \log \frac{\det g(t)}{\Omega^2}= \beta_\Omega-\beta(g(t)),$$
 we obtain
 \begin{equation}
 \frac{\D}{\D t}(g(t)-\hat{g}(t)-\na  d\f (t))=  -\beta(g(t))+\beta_\Omega + \beta (g(t))-\beta_{\Omega}= 0. 
 \end{equation}
 At $t=0$ we have $g(0)-\hat{g}(0) -\na d\f(0)=0$,
 so $g(t)-\hat{g}(t) -\na d\f(t)=0$ on $[0,T')\times M$ as required. 
 \end{proof}
 The uniqueness in Theorem \ref{thm:maximal_time}  now follows from  the following comparison principle for parabolic Monge-Amp\`ere equations:

\begin{proposition} Let $(\hat g(t))_{t\in [0,S]}$ be a smooth family of Riemannian  metrics on $M$. 
Suppose that $\f, \psi\in C^{\infty}([0,S]\times M)$ satisfy $ \hat{g}(t)+\na d\f(t)>0$, $ \hat{g}(t)+\na d \psi(t)>0$ and 
\begin{eqnarray}
\frac{\D\f}{\D t}&\leq& \log \frac{\det(\hat g+\na d \f)}{\Omega^2}-F(t,x,\f),\\
\frac{\D\psi}{\D t}&\geq& \log \frac{\det(\hat g+\na d \psi)}{\Omega^2}-F(t,x,\psi),
\end{eqnarray}
where $F(t,x,s)$ be a smooth function with $\frac{\D F}{\D s}\geq -\lambda$. Then
\begin{equation}
\sup_{[0,S]\times M}(\f-\psi)\leq e^{\lambda S}\max \left\lbrace \sup_X(\f_0-\psi_0); 0 \right\rbrace.
\end{equation}

\begin{proof}
The proof follows  from the maximum principle. Fix $\epsilon>0$, define
$u(t,x)= e^{-\lambda t}(\f-\psi) -\epsilon t.$ Suppose that $u$ achieve its maximum at $(t_0,x_0)\in [0,S]\times  M$. We assume that $t_0>0$, otherwise we are done. At $(t_0,x_0)$, we have $\dot u\geq 0$ and $\na d u\geq 0$, hence
$$-\lambda e^{-\lambda t}(\f-\psi)+ e^{\lambda t}(\dot \f -\dot \psi)\geq \epsilon>0$$
and $\na d \f\leq \na d \psi.$
Therefore at $(t_0,x_0)$ we have
$$\dot \f -\dot \psi\leq - F(t,z,\f)+ F(t,x,\psi ), $$
and $\dot \f -\dot \psi > \lambda (\f-\psi) $, so
$$F(t_0,x_0,\psi(t_0,x_0))+\lambda \psi(t_0,x_0)> F(t_0,x_0,\f(t_0,x_0))+\lambda \f(t_0,x_0).$$
Since $\D F/\D s \geq -\lambda $, $s\mapsto F(\cdot,\cdot, s )$ is  increasing, it comes $\f(t_0,x_0)\leq \psi(t_0,x_0)$. Therefore $u(t,x)\leq u(t_0,x_0)\leq 0$. Letting $\epsilon\rightarrow 0$, gives
\begin{equation}
\sup_{[0,S]\times M}(\f-\psi)\leq e^{\lambda S}\max \left\lbrace \sup_X(\f_0-\psi_0); 0 \right\rbrace
\end{equation}
as required. 
\end{proof}

\end{proposition}
 \subsection{$C^0$ and $C^1$ Estimates}
 We now assume that the solution $\f$ of the parabolic Monge-Amp\`ere equation \eqref{PMA_max} on $[0,T_m)$ for $0<T_m<T'<T$. We shall establish uniform estimates for $\f$ on $[0, T_m)$. 
 The estimates for $\f$ and $\dot \f$ follow from the maximum principle.
 \begin{lemma}\label{lem:C0_max_time}
 There is a  uniform constant $C>0$ such that $\sup_M |\f (t)|\leq C,$ for all $t\in [0,T_m)$. 
 \end{lemma}
 \begin{proof}
 For the upper bound of $\f$ we apply the maximum principle to $H= \f- A\f$ for 
 $$A= 1+ \sup_{[0,T_m]\times M}\log \dfrac{\det \hat g}{\Omega^2},$$
where $\hat g(t) = g_0+ \frac{t}{T'}(g'-g_0)$.  For any $s\in [0,T_m)$,  suppose that
 $$\f(t_0,x_0)=\max_{[0,s]\times M}\f ,$$
 with $(t_0,x_0)\in [0,s]\times M$.  If $t_0>0$, then using the fact that $(\f_{ij}(t_0,x_0)) \leq 0$, we have
 \begin{equation*}
     0\leq \frac{\D \f}{\D t} =\log \dfrac{\det(\hat g+ \na d \f)}{\Omega^2} -A \leq \log \dfrac{\det(\hat g) }{\Omega^2} -A\leq -1,
 \end{equation*}
 so a contradiction. Hence $t_0=0$ and we get the upper bound for $\f$.  

\medskip
 Similarly, we use the same argument to $K= \f+ Bt$, where
 $$B= 1-\inf_{[0,T_m]\times M}\log \dfrac{\det \hat g}{\Omega^2},$$
  to get the lower bound of $\f$. 
 \end{proof}
We shall use the following evolution equation for $\bar\beta(g):=\tr_g\beta(g)= g^{ji}\beta_{ij}(g)$.
\begin{proposition}\label{thm:evol_trace_beta}
 The trace $\bar \beta $ of $\beta$ evolves by
 \begin{equation}
     \frac{\partial }{\partial t} \bar \beta = L_g \bar\beta +|\beta|_g^2,
 \end{equation}
 where
 $|\beta|_g^2 = g^{li}g^{jk} \beta_{ij}\beta_{kl}$. Therefore we have  lower bound
 $$\bar \beta \geq  -Ce^{-t},$$
 where $C= -\inf_{M}\bar \beta(0)-n$. 
\end{proposition} 
\begin{proof}
It follows from the flow that
\begin{eqnarray}\label{eq:flow_inverse}
\frac{\partial}{\partial t}g^{ij}&= &-g^{il}(\partial_t g_{lk})g^{kj}\\
&=& -g^{il}(-\beta_{lk})g^{kj} \\
&=& \beta^{ij},
\end{eqnarray}
where $(g^{kl})=(g^{-1})_{kl}$. 
Taking the trace both two sides of the flow \eqref{HKF_max}, we get
\begin{eqnarray} \label{eq:trace}
g^{lk}\partial_t g_{kl} =\bar \beta . 
\end{eqnarray}

Combining \eqref{eq:trace}, \eqref{eq:flow_inverse} and  $\beta_{ij}= \partial_i \partial _j \log \det [g_{kl}]$ yields
\begin{eqnarray} 
\partial_t \bar \beta &=& g^{ji }\partial_i \partial_j  (g^{kl}\partial _t g_{kl}) - \partial_{t} g^{ji} \beta_{ij}\\
&=& g^{ji }\partial_i \partial_j  (\bar{\beta}-n) - \partial_{t} g^{ji} \beta_{ij}\\
&=&  L_g \bar \beta   -  \beta^{ji}\beta_{ij}\\
&=& L_g \bar \beta   - g^{li} g^{jk}  \beta_{kl} \beta_{ij} \\
&=& L_g \bar \beta   + | \beta|_g^2
\end{eqnarray}
as required. 
\end{proof}
 
 \begin{lemma}\label{lem:bound_phi_dot_max_time}
 There is a uniform constant $C$ 
 such that
 $$\sup_M|\dot \f(t)|\leq C.$$
 for all $t\in [0,T_m)$. 
 
As a consequence,  we have $e^{-C}\Omega^2 \leq  {\det g(t)} \leq e^C \Omega^2$, so $C'^{-1}\det {g_0} \leq  {\det g(t)} \leq C' \det g_0$ for some constant $C'$ depending only on $g_0$ and $\Omega$. 
 \end{lemma}
 
 \begin{proof}
 We first have
 $$\dfrac{\D }{\D t}\dot \f =L_g \dot \f -\tr_g \beta_\omega, $$
 since $\beta_\Omega=\frac{-1}{T}(g'-g_0)$ and $\hat{g}(t)= g_0-t\beta_{\Omega}$, where  $L_g=\tr_g \na d$. 
 Therefore we have
 \begin{equation}
  \left( \frac{\D }{\D t}-L_g\right) ((T'-t)\dot{\f} )= -\dot \f -(T'-t)\tr_g \beta_\Omega.
 \end{equation}
We also have
\begin{equation}
  \left( \frac{\D }{\D t}-L_g\right) \f = \dot \f + \tr_g (g-\hat g)=\dot \f- \tr_g\hat{g}- n.
 \end{equation}
Let  $H= (T'-t)\dot \f + \f +nt$, then combining identities above gives
 \begin{eqnarray}
  \left( \frac{\D }{\D t}-L_g\right)H = \tr_{g}( -(T'-t)\beta_\Omega +  \hat{g}_t)=\tr_{g}\hat{g}_{T'}>0.
 \end{eqnarray}
 The maximum principle then implies that the minimum of $H$ is at $ t=0$. Therefore 
 $$(T'-t)\dot \f(t)  +nt\geq T'\dot{ \f}(0)\geq T'\inf_M \log\frac{\det g_0}{\Omega^2}\geq -CT', $$
so using Lemma \ref{lem:C0_max_time} and $T'-t>T'-T_m>0$ gives 
 $$\inf_M\dot \f(t)\geq -C,$$
for all $t\in [0,T_m)$. 

For the upper bound of $\dot \f $, we observe that
\begin{equation} \label{eq:phi''}
\D_t \dot\f=-\tr_g\beta(g)=-\bar \beta(g). 
\end{equation}
It follows from Proposition \ref{thm:evol_trace_beta} that,  along the flow, $\bar \beta$ satisfies 
\begin{equation}
    \frac{\partial }{\partial t} \bar \beta = L_g \bar\beta +|\beta|_g^2 ,
\end{equation} 
and so  $$\inf_M \bar\beta(t)\geq \inf_M \bar{\beta}(0)\geq -C,$$ for all $t\in [0,T_m)$. Combining with \eqref{eq:phi''} gives a uniform upper bound for $\dot \f$ for all $t\in [0,T_m)$ as required.
 \end{proof}
 \subsection{$C^2$-Estimate}
 Our goal in this section so to prove the following  estimate:
 \begin{theorem}\label{thm:C2_max_time}
     There exists a constant $C>0$ which depends only on $g_0$ such that 
     $$\sup_M \tr_{g_0}g\leq C,$$
     for all $t\in [0,T)$.
     
     Moreover, there  exists a constant $C'$ depending only on $g_0$  such that
     \begin{equation}     
     \frac{1}{C'} g_0\leq g(t)\leq C'g_0,
     \end{equation}
     for all $t\in [0,T_m)$.
 \end{theorem}
We first prove the following estimate for the trace of the metric along the Hesse-Koszul flow which can be seen as a real version of  the one of the K\"ahler-Ricci flow due to Cao \cite{Cao} (see \cite{SW, Tos} for details). 

\begin{lemma}\label{lem:evol_log_tr}
Let $g_0$ be a fixed Hessian metric on $M$. There exists a uniform constant $C_0$ depending only on the metric $g_0$ such that the solution $g(t)$ of the Hesse-Koszul flow satisfies 
  \begin{eqnarray}
  \left( \frac{\D}{\D t}-L_g \right) \log \tr_{g_0}g \leq C_0   \tr_g g_0.  \end{eqnarray}

\end{lemma}

 \begin{proof}
 We now calculate at a point $x_0$ with affine coordinates $\{x^1, \ldots, x^n\}$ for  $\na$ such that $(g_0)_{ij}(x_0)=\delta_{ij}$ and $\D_k (g_0)_{ij}(x_0)=0$. At $x_0$, for $L_g= \tr_g \na d$, we have
 \begin{eqnarray}
  L_g(\tr_{g_0}g )&=& g^{ij}\D_i \D_j (g_0^{kl} g_{kl})\\
  &=& g^{ij}g_0^{kl}\D_i\D_j g_{kl} + g^{ij} \D_i\D_j g_{0}^{kl}g_{kl}.   
 \end{eqnarray}
Observe that
 $$g^{ij} \D_i\D_j g_{0}^{kl}g_{kl} \geq -C_0 \tr_{g_0} g \tr_g g_0,$$
 for  a constant $C_0>0$ depending only on $g_0$. 
We also have
 \begin{eqnarray}
  \frac{\D }{\D t}\tr_{g_0}g &=& g_0^{lk}\D_k\D_l \log \det (g) \\
  &= &g^{ji} g_0^{kl}\D_k\D_l g_{ji} -g^{jp}g^{qi}g_0^{kl} \D_k g_{ij}\D_l g_{pq}\\
  &=& g^{ji}g_0^{kl}\D_i\D_j g_{kl}  -g^{jp}g^{qi}g_0^{kl} \D_k g_{ij}\D_l g_{pq}.
 \end{eqnarray}
 This implies that
 
 \begin{eqnarray}
  \left( \D_t-L_g \right) \tr_{g_0}g \leq  C_0 \tr_{g_0} g  \tr_g g_0 -g^{jp}g^{qi} g_0^{lk}\D_k g_{ij}\D_l g_{pq},  \end{eqnarray}
 hence
  \begin{eqnarray}\label{ineq_trace}
  \left( \D_t-L_g \right) \log \tr_{g_0}g \leq C_0   \tr_g g_0 + \dfrac{1 }{\tr_{g_0}g} \left( - g^{jp}g^{qi} g_0^{lk}\D_i g_{kj}\D_l g_{pq} +\frac{g^{qk }\D_k\tr_{g_0}g \D_q\tr_{g_0}g}{\tr_{g_0}g} \right). \end{eqnarray}
  We now claim that the second term  satisfies
  \begin{equation}\label{ineq_trace_2}
        - g^{jp}g^{qi} g_0^{lk}\D_i g_{kj}\D_l g_{pq} +\frac{g^{lk }\D_k\tr_{g_0}g \D_l\tr_{g_0}g}{\tr_{g_0}g} \leq 0. 
  \end{equation}
  Indeed, since 
  $$g_0^{li}g^{jp}g^{qk} A_{ijk}A_{lpq}\geq 0, $$
  where 
  $$A_{ijk}= \D_i g_{kj}-g_{ij}\frac{\D_k \tr_{g_0}g}{\tr_{g_0}g},$$
  we get
  \begin{eqnarray*}
   0&\leq&  g^{jp}g^{qk} g_0^{li}\D_i g_{kj}\D_l g_{pq} + g_0^{li } g^{jp}g^{qk} g_{ij}g_{pl}   \dfrac{\D_k\tr_{g_0}g \D_q\tr_{g_0}g  }{\tr^2_{g_0}g} -2 g_0^{li}  g^{jp}g^{qk} 
   g_{lp} \D_k g_{il} \frac{\D_q \tr_{g_0} g }{\tr_{g_0}g}\\
   &=& g^{jp}g^{qk} g_0^{li}\D_i g_{kj}\D_l g_{pq}  +  g^{qk}  g_0^{li } g_{il}   \dfrac{\D_k\tr_{g_0}g \D_q\tr_{g_0}g  }{\tr^2_{g_0}g} - 2 g_0^{li}  \delta_{jl}g^{qk} 
 \D_k g_{il} \frac{\D_q \tr_{g_0} g }{\tr_{g_0}g}\\
 &=& g^{jp}g^{qk} g_0^{li}\D_i g_{kj}\D_l g_{pq}  + \frac{g^{qk }\D_k\tr_{g_0}g \D_q\tr_{g_0}g}{\tr_{g_0}g} -2 g_0^{li} g^{qk} \D_k g_{il} \frac{\D_q \tr_{g_0} g }{\tr_{g_0}g}.
  \end{eqnarray*}  
  Therefore we have
  \begin{eqnarray}
   0\leq  g^{jp}g^{qi} g_0^{lk}\D_i g_{kj}\D_l g_{pq} + \frac{g^{qk }\D_k\tr_{g_0}g \D_q\tr_{g_0}g}{\tr_{g_0}g} -2 g_0^{li} g^{qk} \D_k g_{il} \frac{\D_q \tr_{g_0} g }{\tr_{g_0}g}.
  \end{eqnarray}  
  Since $g_0^{li}\D_k g_{il}= \D_k \tr_{g_0}g -g_{il} \D_k g_0^{li} =\D_k \tr_{g_0}g$, hence
  
  \begin{eqnarray}
   0\leq  g^{jp}g^{qi} g_0^{lk}\D_i g_{kj}\D_l g_{pq} - \frac{g^{qk }\D_k\tr_{g_0}g \D_q\tr_{g_0}g}{\tr_{g_0}g},
  \end{eqnarray}  
  as required. The desired inequality now follows from \eqref{ineq_trace}
 and \eqref{ineq_trace_2}.
   \end{proof}
  
    \begin{proof}[Proof of Theorem \ref{thm:C2_max_time}] 
It follows form Lemma \ref{lem:evol_log_tr} that 
   \begin{eqnarray}
 \left(\frac{\D}{\D_t} -L_g \right) \log \tr_{g_0}g \leq C_0   \tr_g g_0 , \end{eqnarray}
 hence 
 \begin{equation}
  \left(\frac{\D}{\D_t} -L_g \right) \left(\log\tr_{g_0}g +C_0(t\dot \f(t)-\f(t)-nt) \right)\leq 0
 \end{equation}
Using the maximum principle, the function $H= \log\tr_{g_0}g +C_0(t\dot \f(t)-\f(t)-nt) $ achieves its maximum at $t=0$. hence
$$\log \tr_{g_0}g\leq C_1-C_0C_0(t\dot \f(t)-\f(t)-nt).$$
 Since $\f$ and $\dot \f$ are uniformly bounded, we imply that $\tr_{g_0}g$ is uniformly bounded from above as required. The second assertion follows from the next Lemma and the fact that $C^{-1} g_0\leq \det g(t)\leq C \det g_0$ (cf. Lemma \ref{lem:bound_phi_dot_max_time}).
\end{proof}

\begin{lemma}\label{lem:ineq_volumes}
If $g_1$ and $g_2$ are two metrics on a compact Riemannian  manifold $M$, then
\begin{equation}
    \left( \dfrac{\det g_2}{\det g_1}\right)^{1/n} \leq \frac{1}{n} \tr _{g_1}g_2\leq \left( \dfrac{\det g_2}{\det g_1}\right) (\tr_{g_2}g_1)^{n-1}.
\end{equation}
In particular, if there exists a constant $C>0$ such that  $C^{-1}\det g_1\leq \det g_2\leq C \det {g_1}$, we have
$$\tr_{g_1} g_2\leq C_1 \Longleftrightarrow \tr_{g_2} g_1\leq C_2\Longleftrightarrow C^{-1}_3g_1 \leq g_2\leq C_3 g_1,$$ 
where for each equivalent relation  $C_i$ depends only on $C$ and $C_j$ with $j\neq i$.
\end{lemma}
\begin{proof}
Let $0<\lambda_1\leq \ldots\leq \lambda_n$ be the eigenvalues of $g_2$ with respect to $g_1$ (at a given point in $M$). The assertion is  now
\begin{equation}
(\lambda_1\ldots\lambda_n)^{1/n}\leq \frac{1}{n}\sum_i\lambda_j\leq  \lambda_1\ldots\lambda_n \left(\sum_j \frac{1}{\lambda_j}\right)^{n-1}.
\end{equation}
The left hand side inequality is nothing but the arithmetico-geometric inequality. For the second one, we can aussume that $\lambda_1\ldots\lambda_n=1$, then 
$$\left(\sum_j \frac{1}{\lambda_j}\right)^{n-1}\geq \lambda_1^{-1}\ldots \lambda_{n-1}^{-1}\geq \frac{1}{n }\sum_j\lambda_j$$
as required. The second claim is straightforward from the first one.
\end{proof}
\subsection{Higher estimates and proof of Theorem \ref{thm:maximal_time}}
We now can use the Evans-Krylov theorem and Schauder estimates to get $C^k$ estimates for all $k\geq 0$
\begin{proposition}\label{prop:Ck_max_time}
For any $k\in \N$, there exists a uniform constant $C_k>0$ such that 
$$\|\f(t)\|_{C^k(M)}\leq C_k,$$
for all $t\in[0,T_m)$.
\end{proposition}

\begin{proof}
Since the Monge-Amp\`ere flow is a fully nonlinear parabolic equation with concave operator, the $C^{2,\alpha}$ estimate for $\f$ follows from  the Evans-Krylov theorem \cite{Kry} (see also \cite{Kry_book, Lie}).
   
Let $D$ be any first order differential operator with constant coefficients. Differentiating the \eqref{PMA_max}, we get
\begin{equation}\label{eq:D}
\left(\frac{\D}{\D t}-L_{g(t)} \right)D \f = D \log (\Omega^2).
\end{equation} 
Since $| \f|+|\dot \f |+| \na d\f |$ and $[\na d \f]_{\alpha}$ are under control, the $\|D \f\|_{C^0}$ is under control.   Applying the parabolic Schauder estimates \cite{Kry_book, Lie} we imply that   $\| Du\|_{C^{2,\alpha}}$ is under control. Applying  $D$ to \eqref{eq:D} we  obtain
\begin{eqnarray}
\left(\frac{\D}{\D t}-L_{g(t)} \right)D^2 \f &=& D^2 \log (\Omega^2) + \sum D g^{jk} \frac{\D ^2 Du}{\D x^j \D x^k },
\end{eqnarray}
where the parabolic $C^\alpha$ norm of the right hand is under control by the argument above. Using again the parabolic Schauder estimates, we obtain uniform bound for $D^2 u$. Iterating this procedure we complete the proof of Proposition \ref{prop:Ck_max_time}.
\end{proof}
\begin{proof}[Proof of Theorem \ref{thm:maximal_time}]
It follows from the Arzel\`a-Ascoli Theorem and Proposition \ref{prop:Ck_max_time} that  given any sequence $t_j\rightarrow T_m $, there exists a subsequence $t_{j_k}$ and a  smooth function $\f_{T_m}$ such that $\f_{t_{j_k}}\rightarrow \f_{T_m}$ in $C^{k}(M,g_0)$ for all $k\geq 0$.  It follows from Lemma \ref{lem:bound_phi_dot_max_time} $\sup_M|\dot \f| \leq C$  for all $t\in [0,T_m)$ for some uniform constant $C$, so $\f(t)-Ct$ is non-increasing in $t$. 
 Since $\f$ is uniformly bounded in $[0,T_m)$,  $\f(t)-Ct$  converges as $t\rightarrow T_m$, to a function which is necessarily equal $\f_{T_m}$. Therefore the limit $\f_{T_m} $ is unique and so $\f(t)\rightarrow \f_{T_m}$ in $C^k(M,g_0) $ for all $k\geq 0$. 
 
 Therefore the metric $g(t)=\hat g(t)+\na d\f$ converges smoothly to the tensor $g_{T_m}=\hat g_{T_m}+\na d \f_{T_m}$. Moreover it follows from Theorem \ref{thm:C2_max_time} that $g(t)\geq C g_0$ for all $t\in [0,T_m)$, so $g_{T_m}$ is positive definite, i.e a Riemannian metric. Therefore the flow can be extended to $t=T_m$, this is a contradiction to our assumption that $T_m$ is the maximal  time of the existence. This implies that  $T_m=T$ as required. 
\end{proof}

 \section{Hesse-Koszul flow and Hesse-Einstein metrics}
We consider the case when $c^a_1(M)<0$. Then by Theorem \ref{thm:maximal_time}, the Hesse-Koszul exists for all time. Since the class $[g(t)]= [g_0]- tc^a_1(M)$ becomes unbounded as $t\rightarrow \infty$,  we can not study the convergence of the flow. Therefore we need to rescale the flow in time by  $t=\log (s+1)$ where $s$ is the time variable for the original flow. Then we get the {\it normalized Hesse-Koszul flow}: 
\begin{equation}\label{eq:NHEF}
    \frac{\partial  g}{\partial t}=-\beta(g) -g, \quad g|_{t=0}= g_0.
\end{equation}
This flow also exists for all time and the   class $[g(t)]$ satisfies 
\begin{equation}
    \dfrac{d}{d t}[g(t)] = -c_1^a(M)-[g(t)],\quad \quad [g(0)] =[g_0]. 
\end{equation}
Therefore we have 
$$[g(t)] =e^{-t}[g_0] +  (1-e^{-t})[c_1^a(M)].$$

which yields the following theorem.
\begin{theorem}\label{thm:convergence_neg}
 Assume  that $c_1^a(M)<0$, then starting  from any Hessian metric $g_0$, the flow \eqref{eq:NHEF} exists for all time and converges in $C^\infty$ to a Hesse-Einstein metric $g_{\infty}$ satisfying
\begin{equation}\label{eq:HE_neg}
\beta(g_{\infty})=-g_{\infty}.
\end{equation} 
Moreover, $g_\infty$ is the unique solution to the Hesse-Einstein equation \eqref{eq:HE_neg}. 
\end{theorem}
The existence and uniqueness of a solution to \eqref{eq:HE_neg} was proved by Cheng-Yau \cite{CY} and   Caffarelli-Viaclovsky \cite{CV}.  Our result give another proof of this result using the parabolic approach.

\medskip
Since $c^a_1(M)<0$, there exists a Hessian metric $\hat g$ such that  its affine K\"ahler form  $\hat g\in -c^a_1(M)$. Fix a volume form $\Omega$ such that  locally $\Omega= \Omega(x)dx^1\wedge \ldots\wedge dx^n$ with
$$ \dfrac{\D^2}{\D x^i \D x^j} \log \Omega^2 (x)= \hat g_{ij}(x) .$$
By the same argument in Lemma \ref{lem:reduction_max}, we can rewrite the normalizied Hesse-Enstein flow as the parabolic Monge-Amp\`ere equation
\begin{equation}
    \dfrac{\D  }{\D t} \f =\log \dfrac{\det (\tilde g + \na d \f) }{ \Omega^2  } - \f, \quad \tilde g + \na d \f>0, \quad \f|_{t=0}=0. 
\end{equation}
where $\tilde g(t,x) = e^{-t}g_0(x)+ (1-e^t)\hat g (x)$. 

We now establish a priori estimates for $\f$ which are independent of $t$. We follow the same strategy  as for the K\"ahler-Ricci flow.  

\begin{lemma}\label{lem:bound Co}
    There exits a uniform constant $C>0$ such that on $[0,\infty)\times M$
    \begin{enumerate}
        \item $|\f(t)|\leq C$.
        \item $| \dot\f(t)|\leq C(t+1)e^{-t}$
        \item There exists a  function $\f_\infty\in C^0(M)$ such that 
        $$|\f(t) - \f_\infty| \leq Ce^{-t/2}.$$
        \item $C^{-1}\Omega \leq  \sqrt{\det g(t)} \leq C \Omega$.
    \end{enumerate}
\end{lemma}
\begin{proof}
The first estimate is derived straightforwardly from the maximum principle applied to $\f$. For (2), we use the maximal principle for the function $H=(e^t -1)\dot \f -\f -nt$.
Taking the derivative in time on the two sides of the flow, we get
\begin{equation}
    \dfrac{\D }{\D t} \dot \f=L_{g(t)} \dot \f + {\rm Tr}_{g(t)} (-e^{-t}g_0+e^{-t}\hat g ) - \dot \f,
\end{equation}
where $L_g$ is the elliptic operator $\tr_g \na d $.
In addition, we have
\begin{equation}
    \left(  \dfrac{\D }{\D t}-L_{g(t)}\right) \f =  \dot \f -{\rm Tr}_{g(t)}(g(t)-\tilde g(t))
    = \dot \f -n + {\rm Tr}_{g(t)}(e^{-t}g_0 +(1-e^{-t}) \hat g).
\end{equation}
Therefore it comes
\begin{equation}
    \left(  \dfrac{\D }{\D t}-L_{g(t)}\right)(e^t \dot \f ) = {\rm Tr} _{g(t)} (-g_0+\hat g)
\end{equation}
and 
\begin{equation}
     \left(  \dfrac{\D }{\D t}-L_{g(t)}\right)(\dot \f +\f +nt )= {\rm Tr}_{g(t)} \hat g.
\end{equation}
Combining all above, we get
\begin{equation*}
     \left(  \dfrac{\D }{\D t}-L_{g(t)}\right) H =-{\rm Tr}_{g(t)}g_0<0.
\end{equation*}
Then the maximum principle implies that $H=(e^t -1)\dot \f -\f -nt \leq 0$.  Since $\f$ is bounded by (1), we get $\dot \f\leq C(t+1)e^{-t}$.  

\medskip
For the lower bound of $\dot \f$ we apply the maximum principle to  
\[
G= (e^t + B)\dot \f+ B\f +nBt
\]
where $B$ satisfies $B\hat g\geq g_0$.  By a direct computation we get
\begin{equation*}
     \left(  \dfrac{\D }{\D t}-L_{g(t)}\right) G = {\rm Tr}_{g(t)} (-g_0+\hat g+ B \hat g)>0.
\end{equation*} 

The maximum principle thus implies that $G\geq 0$, hence  $\dot \f\geq -C(t+1)e^{-t}$. 
For (3),  taking $s>t$ and $x\in M$, and using (2), we get
\begin{eqnarray}
|\f(s,x)-\f(t,x)|&=&|\int_t^s\dot \f(\tau,x)d\tau|\\
&\leq& C\int_t^s (1+\tau)e^{-\tau} \leq C\int_t^se^{-\tau/2}d\tau =2C(e^{-t/2}-e^{-s/2}).
\end{eqnarray}
Therefore $\f(t)$ converges uniformly to a continuous function $\f_\infty$. Letting $s\rightarrow \infty$ yields (3). 

Finally, we use (1) and (2) for 
$$\log\dfrac{\det[g(t)]}{\Omega^2} =\dot \f +\f$$ 
to get (4). 
\end{proof}

\begin{proposition}\label{lem:trace}
    There exists a constant $C$ such that 
    $$C^{-1}g_0\leq g(t)\leq C g_0, \quad \text{on}\quad [0,\infty)\times M.$$
\end{proposition}

\begin{proof} It follows from Lemma \ref{lem:bound Co}-(3) that there exists a uniform constant $C$ such that
\begin{equation}\label{eq_det_1}
C^{-1}\det g_0\leq \det g\leq C\det g_0.
\end{equation}  Then
using  Lemma \ref{lem:ineq_volumes} it suffices to derive a uniform upper bound for $\tr _{g_0}g$.  

We now  apply the maximum principle to 
 $K= \log \tr _{g_0} g -B\f$, where $B$ is chosen hereafter.  It follows from the proof of  Lemma \ref{lem:evol_log_tr} that 
\begin{equation}
     \left(  \dfrac{\D }{\D t}-L_{g}\right) \log\tr_{g_0}g \leq C_0 \tr_{g}g_0-1 ,
\end{equation}
where $C_0$  only depends  on $g_0$. Please note that here the constant $1$ appears in the right-hand side because of the normalization of the flow. Therefore 
\begin{eqnarray}
 \left(  \dfrac{\D }{\D t}-L_{g}\right)K &\leq& C_0\tr_{g}g_0-1 - B \dot \f +B\tr_{g} (g-g_0)\\
 &= & (C_0-B)\tr_{g}g_0  -B\dot \f + B n-1. 
\end{eqnarray}
Choosing $B=C_0+1$, we have 
\begin{eqnarray}
 \left(  \dfrac{\D }{\D t}-L_{g}\right)K \leq -\tr_{g}g_0 -(C_0+1)\dot \f + (C_0+1) n-1
\end{eqnarray}

Suppose that $K$ admits a maximum at $(t_0,x_0)$ with $t_0>0$, then  at this point
$$- \tr_{g}g_0  -( C_0+1) \dot \f + (C_0+1) n-1 \geq 0 .$$ Since  $\dot \f$ is uniformly bounded by Lemma \ref{lem:bound Co},  we have $ \tr_{g}g_0 \leq C_1$ at $(t_0,x_0)$,   for some uniform constant $C_1>0$. Using Lemma \ref{lem:ineq_volumes} and \eqref{eq_det_1}, we get $\tr_{g_0}g (t_0,x_0)\leq C_2$. By our assumption;  $K\leq K(t_0,x_0)$ hence
\begin{eqnarray}
    \log \tr_{g_0} g(t,x)&\leq& \log \tr_{g_0}g(t_0,x_0)- (C_0+1)(\f(t,x)-\f(t_0,x))\\
    &\leq & C_2 - (C_0+1)(\f(t,x)-\f(t_0,x)).
\end{eqnarray}

Since $\f$ is uniformly bounded by Lemma \ref{lem:bound Co}, this implies that  
 \begin{equation}
     \tr_{g_0} g\leq  C,
 \end{equation}
as required.
\end{proof}
We now can use the Evans-Krylov and Schauder estimates as  Proposition \ref{prop:Ck_max_time} to get $C^k$ estimates for all $k\geq 0$
\begin{proposition}\label{pop:Ck}
For any $k\in \N$, there exists a uniform constant $C_k>0$ such that on $[0,\infty)$ 
$$\|\f\|_{C^k(M)}\leq C_k. $$
\end{proposition}
\begin{proof}[Proof of Theorem \ref{thm:convergence_neg}] Now we can complete the proof of Theorem \ref{thm:convergence_neg}. It follows from Lemma \ref{lem:bound Co} that $\f (t)$ converges uniformly exponentially fast to $\f_\infty$. Since we have all uniform $C^k$ estimates by Proposition \ref{pop:Ck}, we imply that $\f(t)$ converges to $\f_\infty$ in $C^\infty$,  so $\f_\infty$ is smooth. 

In addition, we have $|\dot \f|\leq C(1+t)e^{-t}$ (cf. Lemma \ref{lem:bound Co}), hence $\dot \f$ converge to $0$ in $C^\infty$. It follows that 
$$\log\dfrac{\det(g_\infty)}{\Omega^2}-\f_\infty =0.$$
Taking $\D_i \D_j$ both two sides
gives $\beta_{ij}(g_\infty)= -(\beta_\Omega)_{ij} - (\f_\infty)_{ij} = -(g_\infty)_ {ij}$, hence $g_\infty$ satisfies
\begin{equation}\label{eq:HE_proof}
    \beta(g_{\infty})=g_\infty.
\end{equation}
Finally  the uniqueness follows from the maximum principle. Indeed, suppose $g_1$  and $g_2$ are two Hesse-Einstein metric in $-c_1^{a}(M)$. Then $g_2=g_1+ \na d \phi $ for some  function $\phi\in C^\infty(M$, so $\beta(g_2)=-g_2=\beta(g_1) -\na d\phi$. Therefore we have
\begin{equation}
\na d \log \frac{\det (g_1+\na d \phi)}{\det g_1}= \na d\phi, 
\end{equation}  
so 
\begin{equation}
 \log \frac{\det (g_1+\na d \phi)}{\det g_1}=\phi+C.
\end{equation}
By considering this equality at the maximum $x_0$  of $\phi+C$ we have 
\begin{equation}
 \phi+C =  \log \frac{\det (g_1+\na d \phi)}{\det g_1}\leq 0,
\end{equation}
so $\phi +C\leq 0$. Similarly   we have $\phi+C\geq 0$ so $\phi+C\equiv 0$, hence $g_1=g_2$.
\end{proof} 

\section{The  Hesse-Koszul flow  and real Calabi theorem}

In this section we study the Hesse-Koszul flow twisted by  $\eta \in c_1^a(M)$

\begin{equation}\label{eq:THKF}
     \frac{\partial  g}{\partial t}=-\beta(g) + \eta, \quad g|_{t=0}=g_0.
\end{equation}
Our goal is prove the following convergence result:

\begin{theorem}\label{thm:conv_calabi}
The flow  \eqref{eq:THKF} exists for all time and converges in $C^\infty$ to a metric $g_\infty $ satisfying 
\begin{equation}\label{eq:Calabi}
    \beta(g_\infty)=\eta
\end{equation}
Moreover, $g_\infty$ is the unique solution to the  equation \eqref{eq:Calabi}.
\end{theorem}
 In \cite{MM}, the author proved  the long time existence by proving  a priori estimates possible depending on time. In this section we derive uniform a priori estimates that allows us to prove the smooth convergence in Theorem \ref{thm:conv_calabi}. 
 
 \subsection{Reduction to parabolic Monge-Amp\`ere equation}
 Since $\eta$ and $\beta(g_0)$ belong to the first affine Chern class, there exists $f \in C^\infty(M) $ such that 
 \begin{equation*}
     \eta= \beta(g_0)+ \na d f. 
 \end{equation*}
 Therefore if we let
 $$g=g_0+\na d \f,$$
 where $\f \in C^{\infty}([0,\infty)\times M)$ and $\f(0,\cdot)=0$, 
 then flow becomes
\begin{eqnarray}
 \na d \left(\dfrac{\D }{\D t} \f \right)&= &-\beta(g) +\eta \\ 
&= & -\beta(g)+\beta(g_0) + \na d f\\
&= & \na d \log\frac{\det (g)}{\det (g_0)} +\na d f \\
&=&\na d \left(  \log\frac{\det (g_0+\na d \f)}{\det (g_0)}) +f \right).
\end{eqnarray}
The maximum principle on compact manifolds implies that solving \eqref{eq:THKF} is equivalent to solving the parabolic Monge-Amp\`ere equation 
\begin{equation}\label{PMA_calabi}
    \dfrac{\D }{\D t} \f  =  \log\frac{\det (g_0+\na d \f)}{\det (g_0)} +f, \quad  g_0+\na d \f>0, \quad \f|_{t=0}=0.  
\end{equation} 
 We start with the following observation:
 
\begin{lemma}\label{lem:bound_dot_phi}
 We have

\begin{enumerate}
    \item  There exists a uniform constant $C_1$ such that 
$$\| \dot \f (t)\|_{C^0(M)}\leq C_1, \, \forall t\in [0,\infty)$$ 

    \item There exists a uniform constant $C_2$ such that on $[0,\infty)$ 
\begin{equation}
    C_2^{-1}\det(g_0)\leq \det(g)\leq C_2\det(g_0). 
\end{equation}
\end{enumerate}  
 \end{lemma}
\begin{proof}
Taking the derivative both two sides of  \eqref{PMA_calabi} with respect to $t$ we get
\begin{equation}
    \dfrac{\D}{\D t} \dot \f = L_{g(t)} \dot \f.
\end{equation}
Then the first estimate follows from the maximum principle and the second estimate follows from the first one and the fact that
$$\det(g)=e^{\dot \f -f }\det (g_0).$$
\end{proof}
Although the uniform estimate for $\dot \f$ is quite straightforward to obtain, it is  important for the $C^0$ estimate and the proof of the convergence of the flow. 

\subsubsection{$C^0$ estimate}
We now prove a uniform $C^0$ estimate using the following parabolic ABP type maximum principle due to Tso \cite{Tso}.
\begin{theorem}\label{thm:ABP_parabolic}
 Let $u\in C^{2,1}((0,T)\times U)$ with $u\leq 0$ on the parabolic boundary $\partial_P ((0,T)\times U))$ and let: $$A_u=\{(t,x):\, u(t,x)\geq 0, \exists \xi, |\xi|\leq Md^{-1},   u(t,x)+ \xi .(y-x)\geq u(x,s),\forall y\in U, s\leq t\},$$ where $M=\max u(t,x)>0$ and $d$ is the diameter of $U$. Then  
 \begin{equation}
     M\leq C_n d^{n/(n+1)}\left( \int_{A_u}|u_t\det (u_{ij})|dxdt \right)^{1/(n+1)},
 \end{equation}
where $C_n$ depends only on $n$. 
\end{theorem}

\begin{proposition}\label{prop:C0_est_Calabi}
    There exists a uniform constant $C>0$ such that 
$$\|\tilde \f\|_{C^0(M)}\leq C $$  where 
    $$\tilde \f:= \f- \dfrac{1}{Vol_{g_0}(M)}\int_M \f dV_{g_0},$$
    and 
    $$ Vol_{g_0}(M):=\int_M  dV_{g_0}.$$
\end{proposition}
\begin{proof}
We first remark that the  set $\mathcal{F}_0=\{u\in C^{\infty}(M): g_0+\na d u>0 ,\, \sup_X u=0 \}$ is relatively compact in $ L^{1}(X, \mu )$ with $\mu=\frac{1}{\int_M dV_{g_0}} dV_{g_0}$ (see for instance \cite[Theorem 3.2.12]{Hor}).  Therefore there exists a constant $C$
such  depending only on $M$ and $g_0$ such that  
$$\sup_M \f \leq \int_M \f d\mu +C,$$
hence 
$\tilde{\f}\leq C$. 

 For the uniform lower bound of $\tilde{\f}$, we  follow idea by  the second author jointly with D. H. Phong  in \cite{PT} to use the parabolic ABP estimate due to Tso (Theorem \ref{thm:ABP_parabolic}). 
Fix now any $T<\infty $, and set for each $t$ 
$$
L=\min_{[0,T]\times M} \tilde \f =\tilde \f (t_0,x_0)
$$
for some $(t_0, x_0)\in [0,T]\times M$. We now show that $L$ is bounded from below by a constant independent of $T$. We can assume that $t_0>0$, otherwise we are already done. Let $(x^1,\cdots,x^n)$ be affine coordinates for $M$ (with respect to $\na$) centered at $x_0$, $B_1=\{x|\,|x|<1\}$, and define the  function 
$$\phi =\tilde{\f}+{\delta^2\over 4}|x|^2+|t-t_0|^2,
$$
on  $U_\delta =B_1\times \{t|\, -\delta\leq 2(t-t_0)<\delta\}$,
where $\delta>0$ is small. Clearly $\phi $ attains its minimum on $U$ at $(t_0, x_0)$, and $\phi \geq \min _{U_\delta}\phi +{1\over 4}\delta^2$ on the parabolic boundary of $U_\delta$. 

\medskip 
Define the set
\begin{equation}
\label{contact set}
S:=\left\lbrace (x,t) \in U_\delta: \begin{array}{l}
  \phi(x,t)\leq \phi(z_0,t_0)+{1\over 4}\delta^2,\quad |D_x \phi(x,t) |<\frac{\delta^2}{8}, \textit{ and }\\
  \phi(y,s)\geq \phi(x,t)+D_x \phi(x,t). (y-x),\, \forall y\in U, s\leq t 
 \end{array}
\right\rbrace.
\end{equation}
 Then applying Theorem \ref{thm:ABP_parabolic} to  the function  $u=-\phi +\min_U\phi +{\delta^2\over 4}$ we obtain 
 $$C\delta^{2n+2}\leq \int_{S} (-\phi_t)\, \det(\phi_{ij})dxdt, $$
   for    a constant $C=C(n)>0$. 
It follows from Lemma  \ref{lem:bound_dot_phi}  that  $|\dot \f|$  and  $\det(\phi_{ij})$ are uniform bounded from above hence 
\begin{equation}
    C\delta^{2n+2}\leq A\int_S dxdt. 
\end{equation}
Next, by the definition of $S$, we have $\phi(x,t)\leq L+{\delta^2\over 4}$. Since we can assume that $|L|>\delta^2$, it follows that $\phi <0$ and $|\phi|\geq {|L|\over 2}$ on $S$.  Therefore

\begin{equation}\label{eq:ABP}
    C\delta^{2n+2}
\leq
A\int_S dxdt
\leq
A{|L|\over 2}\int_S |\phi(x,t)|dxdt
\leq
A{|L|\over 2}\int_{U_\delta} |\phi(x,t)|dxdt.
\end{equation}

On $S$, we  also have $$
|\phi|
=-\phi =-\tilde{\f}-{\delta^2\over 4} |x|^2-(t-t_0)^2
\leq -\tilde \f+ \sup_M \tilde{\f}. 
$$
since ${\rm sup}_X\tilde{\f}\geq 0$. Combining with \eqref{eq:ABP}  gives 

\begin{equation} \label{ineq_C0_calabi}
C\delta^{2n+2}
\leq A {|L|\over 2}\int_{|t|<{1\over 2}\delta}\|\tilde \f - {\rm sup}_M \tilde{\f}\|_{L^1(M, g_0)} dt.
\end{equation}
Since  $\mathcal{F}_0$ is relatively compact in $ L^{1}(X, \mu )$, there is a uniform constant $C>0$ such that  $\|\tilde \f - {\rm sup}_M \tilde{\f}\|_{L^1(M, g_0)} \leq C$. Therefore we get
$$C\delta^{2n+2}
\leq A {|L|\over 2}\int_{|t|<{1\over 2}\delta}\|\tilde \f - {\rm sup}_M \tilde{\f}\|_{L^1(M, g_0)} dt
\leq A'\delta {|L|\over 2}$$
from which the desired bound for $L$ follows.
\end{proof}
\noindent
{\bf Remark.} Our approach here can be applied to the K\"ahler-Ricci flow to get a new  $C^0$ estimate which is different from the initial approach of Cao \cite{Cao} for the K\"ahler Ricci flow. 
\subsection{$C^0$ estimate for elliptic Monge-Amp\`ere equations}
We now give another $C^0$ estimate for the flow using a uniform $C^0$ estimate for elliptic Monge-Amp\`ere equation
\begin{equation}
   \det (g_0+\na d \phi)=f \det(g_0)
\end{equation}

The idea of proof comes from the one of $C^0$ estimate  for the Complex Monge-Amp\`ere equation using the Aleksandrov-Bakelman-Pucci estimate  by Blocki  \cite{Bl}. This approach  originated in the work of Cheng-Yau (cf. \cite{Bed})  and  is recently   revisted  by Blocki  \cite{Bl, Bl2} and Sz\'ekelyhidi  \cite{Sze}.

We shall use the following ABP type estimate (cf. \cite[Lemma 9.2]{GT} and \cite[Proposition 11]{Sze}). 
\begin{proposition}\label{prop:ABP}
Let $u: \mathbb{B}\rightarrow \mathbb{R}$ be a smooth function such that $u(0)+\epsilon\leq \inf_{\mathbb{B}} u$ where $\mathbb{B}$ denotes the unit ball in $\mathbb{R}^n$.  Define the contact set
\begin{equation}
S=\left\lbrace
 x\in \mathbb{B}: |D u(x)|\leq \epsilon/2 \text{ and }\,
u(y)\geq u(x)+D u(x).(y-x), \forall y \in \mathbb{B} 
  \right\rbrace. 
\end{equation}
Then there exists a dimensional  constant $C_n>0$ such that
\begin{equation}
C_n\epsilon^n\leq \int_S \det(D^2 u). 
\end{equation}
\end{proposition}
\begin{theorem}\label{thm:C0_estimate}
 Let $(M,\na,g)$ be a compact Hessian manifold of dimension $n$ and   $\phi\in C^2(M)$ satisfying $$\quad g+\na d \phi> 0,\quad \det( g+\na d \phi)= f\det g.$$ 
 
 There exists a uniform constant $C$ depending only on $\|f\|_{L^\infty(M)}$ and $g_0$ such that 
 $$ {\rm osc}_M\phi :=\sup_M \phi -\inf_M \phi  \leq C.$$
\end{theorem}

\begin{proof}
By adding some constant, we assume that $\sup_M \phi=0$.
It  suffices to  derive a uniform bound for $I= \inf_M \phi$.  

 We fix an affine coordinates $(U, x^1,\ldots,x^n)$ with $U=\{x\in \mathbb{R}^n :  |x^i|<1, \, \forall \, i=1,\ldots,n \}$ for which  $I$ is achieved at the origin. Let $\psi= \phi+ \epsilon |x|^2$ for some small $\epsilon>0$. Then on the unit ball $\mathbb{B}=\{x:  |x|<1\}$,   $I= \inf_{\mathbb{B}}\psi =\psi(0) $  and $ \psi(x)\geq L+\epsilon$ for $x\in \D \mathbb{B}$.   It follows from Proposition \ref{prop:ABP}
  that  
\begin{equation}
C_n \epsilon^{n}\leq \int_S \det (\psi_{ij}). 
\end{equation}
Since on $S$ we have  $(\psi_{ij})=(\phi_{ij}+\epsilon\delta_{ij})\geq 0$ and $\det(g +\na d\phi) =f\det g$, we imply that $\det (\psi_{ij})$ is uniformly bounded from above on $S$. Therefore
\begin{equation}
\label{eq:ineq_vol_contact_set_1}
C_n \epsilon^{n}\leq \int_S \det (\psi_{ij}) \leq C |S|, 
\end{equation}
where $|S|$ denotes the Lebesgue measure of $S$. 

By the definition of $ S$, we have $\psi(0)\geq \psi(x) -\epsilon/2$ and $\psi(x)\leq I +\epsilon/2 <0$.  Therefore
\begin{equation}\label{eq:ineq_vol_contact_set_2}
Vol(S)\leq \frac{\| \psi \|_{L^1}}{|I+ \epsilon/2| }.
\end{equation}
Since $\mathcal{F}_0=\{u\in C^{\infty}(M): g_0+\na d u>0 ,\sup_X u=0 \}$ is relatively compact in $ L^{1}(X, \mu )$ with $\mu=\frac{1}{\int_M dV_{g_0}} dV_{g_0}$, we have $\|\psi\|_{L^1}$ is uniformly bounded. The inequalities \eqref{eq:ineq_vol_contact_set_1} and \eqref{eq:ineq_vol_contact_set_2} thus give a uniform bound for $I$. 
\end{proof}

We now can apply this theorem to the parabolic Monge-Amp\`ere equation \eqref{PMA_calabi}.
\begin{proof}[Another proof for Proposition \ref{prop:C0_est_Calabi}]
We rewrite the flow as
$$\det(g_0+\na d \f)=e^{\dot \f- f}\det g_0$$
Since $ \dot \f$ is uniformly bounded (Lemma \ref{lem:bound_dot_phi}), we can apply Theorem \ref{thm:C0_estimate} to  obtain the uniform estimate for the oscillation of $\f$.   

Next we have  $\int_M \tilde \f dV_g=0$, thus $\inf_M \tilde{\f} \leq 0 $ and $\sup_M \tilde  \f\geq 0$. Therefore the uniform estimate for the oscillation of $\f$  gives a uniform bound for $|\tilde{ \f}|$. 
\end{proof}

\subsection{Higher order estimates}\label{sec:higher_est_Calabi}
We now prove a $C^2$ estimate. Note that there is a difficulty here since we have only a uniform bound for the oscillation of $\f$ but not on $\f$. To overcome this, we shall use the maximum principle for a test function which contains $\tilde{\f}$ instead of  $\f$.

\begin{lemma}\label{lem:C2_Calabi}
There exists  uniform positive  constants $C$ such that 
\begin{equation}\label{eq:trace_est_calabi}
    \tr_{g_0} g(t,x)\leq C.
\end{equation}
In consequence, there exists a uniform $C$ such that
\begin{equation}\label{eq:compare_metrics_calabi}
    C^{-1}g_0\leq g\leq Cg_0.
\end{equation}
\end{lemma}
\begin{proof}
To prove \eqref{eq:trace_est_calabi} we apply the maximal principle to $G= \log \tr_{g_0}g -B\tilde\f  $. 
 It follows from the proof of  Lemma \ref{lem:evol_log_tr} that 
\begin{eqnarray}
     \left(  \dfrac{\D }{\D t}-L_{g}\right)  G \leq C_0\tr_{g}g_0 - \dot{\tilde{\f}} +B\tr_{g} (g-g_0)  +\frac{\tr_{g_0}\eta}{\tr_{g_0}g}.
\end{eqnarray}
where $C_0$  only depends  on $g_0$.  Suppose that $\eta \leq C g_0$ and $\tr_{g_0} g\geq 1$ (otherwise we have done), then choosing $B=C_0+1$ and using the fact that $\dot \f $ is uniformly bounded (cf. Lemma \ref{lem:bound_dot_phi}) we have
\begin{eqnarray}
  \left(  \dfrac{\D }{\D t}-L_{g}\right)G\leq   - \tr_{g}g_0+C
\end{eqnarray}
Suppose that $G$ admits it maximum at $(t_0,x_0)$ at the point $(t_0,x_0)$ and assume without of loss of generality that $t_0>0$. Then the maximum principle implies that $\tr_{g}g_0\leq C$. Using again Lemma \ref{lem:ineq_volumes} gives  $\tr_{g_0}g \leq C$.  Then for any $(t,x)\in [0,\infty)\times M$, we have
$ \log \tr_{g_0}g (t,x) -B \tilde \f(t,x)\leq \log C-B\tilde \f(x_0,t_0)$, thus \eqref{eq:trace_est_calabi} follows from the uniform estimate for $\tilde{\f}$ (cf. Proposition \ref{prop:C0_est_Calabi}). 
\end{proof}

Again we  can use the Evans-Krylov and Schauder estimates to get $C^k$ estimates. 
\begin{proposition}\label{prop:Ck_Calabi}
For any $k\in \N$, there exists a uniform constant $C_k>0$ such that on $[0,\infty)$ 
$$\|\tilde{ \f} \|_{C^k(M)}\leq C_k. $$
\end{proposition}
\subsection{Convergence}\label{sec:convergence_Calabi}
We now finish the proof for the convergence of the flow following the same argument in Cao \cite{Cao}. Set $\psi = \dot \f + A$ for some large constant $A>0$ such that $\psi>0$. Then $\psi$ satisfies the heat equation 
\begin{equation}\label{eq:heat}
    \dot \psi = g^{ij} \D_i\D_j \psi.
\end{equation}
This is uniformly elliptic by the uniformly estimates for $g$ in the previous section. It follows from a straightforward modification of  the  Li-Yau  Harnack inequality for Heat equations (cf. \cite{LY}) that  there exist positive constants $C_1, C_2,C_3$ depending on ellipticity bounds, such that for $0<t_1<t_2$ we have 
\begin{equation} \label{Harnack}
    \sup_M \psi (t_1,\cdot)\leq \inf_M \psi (t_2,\cdot) \left( \frac{t_2}{t_1}\right)^{C_3} \exp\left( \frac{C_2}{t_2-t_1}+C_1(t_2-t_1) \right).
\end{equation}
Using this  inequality we  imply that there  exist positive constants $C_4$ and $\alpha$ so that
\begin{equation}\label{osc_dotphi}
    {\rm osc}_M \psi (t, \cdot)\leq C_4 e^{-\alpha t}.
\end{equation}
Indeed, we define
\begin{eqnarray}
u_s(t,x)&=&\sup_{x\in M} \psi(s-1,x)-\psi (s-1+t,x)\\
v_s(t,x) &=&\psi(s,x )-\inf_{x\in M}\psi (s-1,x)\\
\omega(t)&=& \sup_{x\in M}\psi (t,x) -\inf_{x\in M}\psi(t,x)=: {\rm osc}_M \psi (t, \cdot).
\end{eqnarray}
Then $u_s$ and $v_s$ satisfy the equation \eqref{eq:heat}. Applying the inequality \eqref{Harnack} with $t_1=1/2$ and $t_2=1$, we obtain
\begin{eqnarray}
\sup_{M} \psi(s-1,x)-\inf_M \psi(s-\frac{1}{2},x)  &\leq& C (\sup_M \psi (s-1, x)-\sup_M \psi(s,x))\\
\sup_{M} \psi(s-\frac{1}{2}),x)-\inf_M \psi(s- 1) &\leq& C (\inf_M \psi (s, x)-\sup_M \psi(s-1,x)),
\end{eqnarray}
where $C>1$ is  independent of $s$.  Therefore we imply that
$$\omega(s-1)+\omega(s-\frac{1}{2})\leq C(\omega(s-1)-\omega(s)),$$
so 
$$\omega(s)\leq \delta \omega(s-1),$$
where $\delta=\frac{C-1}{C}<1$. By induction we imply that $\omega(t)\leq Ce^{-\alpha t},$ where $\alpha= -\log \delta$,  as required.

\medskip
It follows from \eqref{osc_dotphi}  that  $|\tilde\psi(t,x)|\leq C_4 e^{-\alpha t}$ for all $x\in M$, where $$\tilde \psi := \psi - \frac{1 }{Vol_{g_o}(M)}\int_M \psi dV_{g_0} = \D_ t \tilde \f,$$
hence 
\begin{equation}
    \D_t (\tilde \f + \frac{C_4}{\alpha} e^{-\alpha t})=\tilde  \psi -C_4 e^{-\alpha t}\leq 0
\end{equation}
and  $\tilde \f+ \frac{C_4}{\alpha}e^{-\alpha t}$ is decreasing in $t$. Since ${\rm osc}_M \f $ is uniformly bounded, so is  $\tilde \f+ \frac{C_4}{\alpha}e^{-\alpha t}$.  
Thus this function converges to a function $\f_\infty$. By the estimates in Proposition \ref{prop:Ck_Calabi},   the function $\tilde \f+ \frac{C_4}{\alpha}e^{-\alpha t}$ converges in $C^\infty$ to $\f_\infty$. Therefore  $\tilde \f(t,x)$ also converges to $\f_\infty$. Now $\tilde \f$ satisfies
\begin{equation}
    \frac{\D}{\D t} \tilde\f +  \frac{1 }{Vol_{g_o}(M)}\int_M\dot  \f  dV_{g_0} = \log \frac{\det (g_0+\na d \tilde \f )}{\det g_0} +f. 
\end{equation}
Letting $t\rightarrow \infty$, we get
\begin{equation}
    \log \frac{\det (g_0+\na d  \f_\infty )}{\det g_0} +f-c = 0,
\end{equation}
where $c= \lim_{t\rightarrow \infty}\frac{1 }{Vol_{g_o}(M)}\int_M \dot \f  dV_{g_0}$. 
It follows that $\beta(g_\infty) =\eta$ as required, where $g_\infty =g_0+\na d \f_\infty $. This completes the proof of Theorem \ref{thm:conv_calabi}.

\subsection{Application of the uniform estimate}
We finish this section to show that our method can be used to  prove the convergence of the following parabolic Monge-Amp\`ere equation on a smooth compact Riemannian manifold:
\begin{equation}\label{PMA_riemannian}
\frac{\D }{\D t}\f (x,t)=\log \frac{\det( g(x)+ \na^2 \f(t,x) )}{\det g(x)}- \lambda \f- f(x),
\end{equation}
with $\lambda=0$,  where $\na$ is the Levi-Civita connection of $g$ and $\lambda\in \mathbb{R}$.  
This follow is studied by Huisken in \cite{Hui} where the author shows that the flow \eqref{PMA_riemannian}  has a longtime existence for all $\lambda\in \mathbb{R}$. She also proved that when $\lambda>0$, the flow converges in $C^\infty$ to a smooth function. The convergence of the flow for $\lambda\leq 0$ is still unknown. In this section, using our approach for uniform  $C^0$  estimate (cf. Proposition \ref{prop:C0_est_Calabi}) we prove that for  $\lambda=0$, the normalization of $\f$ also converges in $C^\infty$ to a smooth function. In particular this result gives an alternative proof to the existence of  solutions for    Monge-Amp\`ere equations on compact Riemannian manifolds due to \cite{Del_JFA}. 

\begin{theorem}
The normalization of the solution $\f$ of the flow \eqref{PMA_riemannian}:
\[
\tilde{\f}:=\f-\frac{1}{Vol_g}\int_M \f dV_g
\]
converges  in $C^\infty$ to a function $\tilde \f_\infty$.  In particular, the limit $\tilde{\f}$ is  a solution 
of the following Monge-Amp\`ere equation:
\begin{equation}
\det(g+\na^2 \phi )=ce^f \det (g). 
\end{equation}
\end{theorem}
\begin{proof}
The key ingredient of the convergence is to prove that $\tilde{\f}$, $\dot \f $  and $\tr_g (g+\na^2 \f) $ is uniformly bounded since all higher uniform estimates are derived from the one of $\tilde{\f}$ and $\dot \f $.

The uniform estimate for  $\dot \f $  is straightforward from  the maximum principle since $\phi=\dot \f$ satisfies the heat equation
\begin{equation}
\frac{\D}{\D t}\phi = \tilde g^{ij}\na_{ij} \phi,
\end{equation}
where $\tilde{g}=g+ \na^2 \phi$.  

For the uniform bound of $\tilde{\f}$ we follow the same argument as in the proof of  Proposition \ref{prop:C0_est_Calabi} using any local coordinates instead of  affine coordinates as in the previous proof.  

For the  uniform estimate for $\tr_g(g+\na \phi) $ we follow the same argument as in Section \ref{sec:higher_est_Calabi}. By the same computation as in Lemma \ref{lem:evol_log_tr},  using normal coordinates for $g$,  we have
  \begin{eqnarray}
  \left( \frac{\D}{\D t}-L_{\tilde g} \right) \log \tr_{g}\tilde g \leq C_0   \tr_{\tilde g} g +C_1,   
  \end{eqnarray}
where $\tilde{g}=g+\na^2\f$,  $C_0 $ depends only on $g$ and $C_1$ depends  only on $f$ and $g$. The constant $C_1$ appears when we compute $\frac{\D}{\D t}\tr_{g}\tilde{g}$.  Applying the maximum principle for $G= \log \tr_g \tilde{g}-(C_0+1) \tilde{\f}$ and following the same   argument in the proof of Lemma \ref{lem:C2_Calabi} we get a uniform bound for $\tr_g \tilde g$ as required. 

Higher order estimates now follow from  Evans-Krylov and Schauder estimates. Finally the convergence result  follows by the same line as in Section \ref{sec:convergence_Calabi}. 
\end{proof}


\begin{thebibliography}{biblio}

\bibitem{Ama}
Amari, S.: Information Geometry and Its Applications. Applied Mathematical Sciences,
vol. 194. Springer, Berlin (2016)
\bibitem{AN}
Amari, S., Nagaoka, H.: Methods of Information Geometry. Translations of Mathematical
Monographs, vol. 191. Am. Math. Soc./Oxford University Press, Providence/London (2000)


\bibitem{As}
P. S. Aspinwall, T. Bridgeland, A. Craw, M. R. Douglas, M. Gross, A. Kapustin, G. W. Moore, G. Segal, B. Szendroi, and P. M. H. Wil- son, Dirichlet branes and mirror symmetry, Clay Mathematics Institute, Cambridge, MA, 2009.


\bibitem{Ay}
Ay, N., Jost, J., Le, H.V., Schwachh\"ofer, L.: Information Geometry, Springer, (2017)

\bibitem{Bed}
Bedford, E.: Survey of pluri-potential theory, in Several complex variables (Stockholm, 1987/1988), 48–97, Math. Notes 38, Princeton Univ. Press, Princeton, NJ, 1993.

\bibitem{Bl}
Błocki, Z.:
On uniform estimate in Calabi-Yau theorem.
Sci. China Ser. A 48 (2005), suppl., 244–247.
\bibitem{Bl2}
Błocki, Z.:
On the uniform estimate in the Calabi-Yau theorem, II, Science China Mathematics, 54 (2011), 1375-1377 
\bibitem{Cao}
Cao, H-D.: Deformation of K\"ahler metrics to K\"ahler-Einstein metrics
on compact K\"ahler manifolds. Invent. Math. 81 (1985), no. 2, 359-
372.
\bibitem{CY}
Cheng, S.-Y., Yau, S.-T.: On the real monge-amp\`ere equation and affine flat structures. In: Proceedings of the 1980 Beijing Symposium on Differential Geometry and Differential Equations, vol. 1-3, pp. 339–370 (1982).
\bibitem{CV}
Caffarelli, L. A.; Viaclovsky, J.  A.: On the regularity of solutions to Monge-Ampère equations on Hessian manifolds. Comm. Partial Differential Equations 26 (2001), no. 11-12, 2339–2351.

\bibitem{Del}
Delano\"e,  P.:  Remarques sur les vari\'et\'es localement hessiennes. Osaka J. Math. 26 (1989), no. 1, 65--69. 

\bibitem{Del_JFA}
Delano\"e, P.:
Équations du type Monge-Ampère sur les variétés riemanniennes compactes. I,  II \& III
J. Functional Analysis 40 (1981), no. 3,  358–386, 341–353, 403–430. 
  \bibitem{Dom}
Dombrowski, P.: On the geometry of the tangent bundle. J. Reine Angew. Math. 210 (1962) 73–88.
\bibitem{Dui}
Duistermaat, J.:  On Hessian Riemannian structures. Asian J. Math. 5 (2001), 79–91. 

\bibitem{GW}
Gross, M.  and  Wilson, P. M. H.:  Large complex structure limits of K3 surfaces, J. Differential Geom., 55 (2000), no. 3, 475–546.

\bibitem{Hit}
Hitchin, N.: The moduli space of special Lagrangian submanifolds, Ann.
Scuola Norm. Sup. Pisa Cl. Sci. (4), 25 (1997), no. 3-4, 503–515

\bibitem{Hor}
H\"ormander, L.: Notions of convexity. Progress in Mathematics, 127. Birkhäuser Boston, Inc., Boston, MA, 1994. viii+414 pp. ISBN: 0-8176-3799-0
\bibitem{Hui}
Huisken, B.: Parabolic Monge-Amp\`ere equations on Riemannian manifolds. J. Funct. Anal. 147 (1997), no. 1, 140–163
\bibitem{HM}
 Hultgren, J. and  \"Onnheim, M.:
An Optimal Transport Approach to Monge–Ampère Equations on Compact Hessian Manifolds, The Journal of Geometric Analysis (2019) 29:1953–1990
%
 \bibitem{Kos61}
 Koszul, J-L.:
 Domaines born\'es homog\`enes et orbites de groupes de transformations affines. Bulletin de la S. M. F., tome 89 (1961), p. 515-533.
 
 \bibitem{Kos62}
Koszul, J-L.: Ouverts convexes homogènes des espaces affines. Math. Z. 79 (1962), 254–259.
 \bibitem{Kos65}

Koszul, J-L.: Variétés localement plates et convexité.  Osaka Math. J. 2 (1965), 285–290.
 \bibitem{Kos68}
Koszul, J-L.:  Déformations de connexions localement plates. (French) Ann. Inst. Fourier (Grenoble) 18 (1968), no. fasc., fasc. 1, 103–114.

\bibitem{Kry}
Krylov,  N. V.: Boundedly inhomogeneous elliptic and parabolic equations, Izv. Akad. Nauk
SSSR Ser. Mat., 46(3) (1982), 487-523.
\bibitem{Kry_book}
Krylov,  N.V.: Lectures on elliptic and parabolic equations in H\"older spaces, Graduate Studies in Mathematics, Vol. 12. American Mathematical Society, Providence, RI, 1996.
\bibitem{KS}
 Kontsevich,  M. and  Soibelman, Y.: Homological mirror symmetry and torus fibrations, in Symplectic geometry and mirror symmetry, 203–263, World Sci. Publishing 2001.

 \bibitem{Lau}
  Lauritzen, S.: Statistical manifolds. {\sl Differential geometry in statistical inference}. Institute of Mathematical Statistics Lecture Notes—Monograph Series, 10. Institute of Mathematical Statistics, Hayward, CA, 1987. iv+240 pp. ISBN: 0-940600-12-9
\bibitem{Lie}
 Lieberman, G.: Second Order Parabolic Differential Equations, World Scientific Publishing, 1996.
 \bibitem{LY}
   Li, P. and Yau, S.-T.: On the parabolic kernel of the Schr\"odinger operator, Acta Math. 156 (1986), no. 3–4, 153–201.
\bibitem{Lof02}
 Loftin, J.: Affine spheres and Kähler-Einstein metrics, Math. Res. Lett. 9 (2002) 425–432.
  \bibitem{Lof09}
  Loftin, J.:   Affine Hermitian-Einstein metrics. Asian J. Math. 13 (2009), no. 1, 101–130.
\bibitem{LYZ}
 Loftin, J.;  Yau,S.-T. and  Zaslow, E.:
Affine manifolds, SYZ geometry and the "Y" vertex, J. Differential Geom.
Volume 71, Number 1 (2005), 129-158.


%
 \bibitem{MM}
 Mirghafouri, M. ; Malek, F.:
Long-time existence of a geometric flow on closed Hessian manifolds. 
J. Geom. Phys. 119 (2017), 54–65.
\bibitem{Pet}
Petersen, P.:  Riemannian Geometry. third edition, GTM 171, Springer (2016)


\bibitem{PT}
Phong, H. D., T\^o, T. D.: Fully non-linear parabolic equations on compact Hermitian manifolds. arXiv:1711.10697. To appear in Annales Scientifiques de l’ENS. 




\bibitem{SYZ}
Strominger, A., Yau, S.-T. , and  Zaslow, E.: Mirror symmetry is T -duality, Nuclear Phys. B, 479 (1996), no. 1-2, 243–259.
\bibitem{Shi}
Shima, H.: The Geometry of Hessian Structures.  World Scientific Pub Co Inc (March 1, 2007)
\bibitem{SW}
Song, J. and Weinkove,B.: An introduction to the Kähler-Ricci flow. {\sl An introduction to the Kähler-Ricci flow}, 89–188, Lecture Notes in Math., 2086, Springer, Cham, 2013.

\bibitem{Sze}
Sz\'ekelyhidi, G.: Fully non-linear elliptic equations on compact Hermitian manifolds , J. Differential Geom. 109 (2018), no. 2, 337--378

\bibitem{Tot}
Totaro, B. The curvature of a Hessian metric. Internat. J. Math. 15 (2004), no. 4, 369–391. 

\bibitem{GT}
Gilbarg, D. and Trudinger, N. S.: Elliptic partial differential equations of second order. Second edition. Grundlehren der Mathematischen Wissenschaften, 224. Springer-Verlag, Berlin, 1983. xiii+513 pp. ISBN: 3-540-13025-X
\bibitem{TZ}
 Tian, G.  and  Zhang, Z.: On the Kähler-Ricci flow of projective manifolds of general
type, Chin. Ann. Math. 27 (2006), no. 2, 179–192.
\bibitem{Tos}
Tosatti,V.: KAWA lecture notes on the K\"ahler-Ricci flow, Ann. Fac. Sci. Toulouse Math. 27 (2018), no.2, 285-376.
\bibitem{TW15}
Tosatti, V. and  Weinkove, B.: On the evolution of a Hermitian metric by its Chern-Ricci form. J. Differential Geom. 99 (2015), no. 1, 125–163
\bibitem{Tso}
Tso,K.: On an Aleksandrov-Bakel’man type maximum principle for second-order parabolic
equations, Comm. Partial Differential Equations 10 (1985), no. 5, 543–553.

\bibitem{Wei}
Weinkove, B.: The K\"ahler-Ricci flow on compact K\"ahler manifolds. Geometric analysis, 53–108, IAS/Park City Math. Ser., 22, Amer. Math. Soc., Providence, RI, 2016.
\bibitem{Yau}
Yau,  S.-T.: On the Ricci curvature of a compact K\"ahler Manifold and the complex Monge-Amp\`ere equation I. Commun. Pure Appl. Math. 31, 339–411 (1978)
\end{thebibliography}
\end{document}